\documentclass{article}
\usepackage[T1]{fontenc}
\usepackage{lmodern}
\usepackage[utf8]{inputenc}
\usepackage{geometry}
\usepackage{booktabs}
\usepackage{color}
\usepackage{amsthm}
\usepackage{amsmath,amssymb}
\usepackage{graphicx}
\usepackage{listings}
\usepackage{emptypage}
\usepackage{newlfont}
\newcommand{\numberset}{\mathbb}
\newcommand{\N}{\numberset{N}}
\newcommand{\Z}{\numberset{Z}}

\DeclareMathOperator{\lc}{span}

\theoremstyle{plain} 
\newtheorem{thm}{Theorem}[section] 
\newtheorem{cor}[thm]{Corollary} 
\newtheorem{lem}[thm]{Lemma} 
 
\newtheorem*{theorem*}{Theorem}

\theoremstyle{definition}

\theoremstyle{definition} 
\newtheorem{rem}[thm]{Remark}

\usepackage{setspace}
\newcommand{\D}{\mathcal{D}}

\title{Greedy expansions with prescribed coefficients in Hilbert spaces for special classes of dictionaries}
\author{Alessandro Oliaro$^1$, Luca Tomatis$^1$, Albert R. Valiullin$^2$, Artur R. Valiullin$^{2,3}$ \\[0.3cm]
	\small{$^1$Dipartimento di Matematica ``G. Peano'', Universit\`a di Torino} \\
	\small{Via Carlo Alberto, 10, I-10123 Torino, Italy.} \\
	\small{$^2$Department of Mathematical Analysis, Faculty of Mechanics and Mathematics,} \\
    \small{Lomonosov Moscow State University, } \\
    \small{Leninskie Gory 1, GSP-1, Moscow, 119991, Russia.} \\
    \small{$^3$Moscow Center for Fundamental and Applied Mathematics, Moscow, Russia.} \\
	\small{E-mail addresses: alessandro.oliaro@unito.it; luca.tomatis143@edu.unito.it;} \\
    \small{albert.valiullin@student.msu.ru; artur.valiullin@student.msu.ru}}
\date{}

\begin{document}
	
\maketitle

\textbf{Abstract}: Greedy expansions with prescribed coefficients have been introduced by V. N. Temlyakov in the frame of Banach spaces. The idea is to choose a sequence of fixed (real) coefficients $\{c_n\}_{n=1}^\infty$ and a fixed set of elements (dictionary) of the Banach space; then, under suitable conditions on the coefficients and the dictionary, it is possible to expand all the elements of the Banach space in series that contain only the fixed coefficients and the elements of the dictionary. In Hilbert spaces the convergence of greedy algorithm with prescribed coefficients is characterized, in the sense that there are necessary and sufficient conditions on the coefficients in order that the algorithm is convergent for all the dictionaries. This paper is concerned with the question if such conditions can be weakened for particular classes of spaces or dictionaries; we prove that this is the case for finite dimensional spaces, and for some classes of dictionaries related to orthonormal sequences in infinite dimensional spaces. \\[0.1cm]

\section{Introduction}

When dealing with Hilbert spaces $H$, the most standard way to expand a generic element of $H$ with respect to a fixed set of vectors $\mathcal{D}\subset H$ is to fix $\mathcal{D}$ as an (orthonormal) basis and consider the usual (unique) expansion with respect to such a basis. Greedy algorithms are different ways of expanding elements in Hilbert spaces, with respect to different fixed sets of vectors, called \emph{dictionaries}. Using such methods it is possible to obtain good properties of the corresponding expansions, related to different aspects; one of the main goals of greedy algorithms is related to the convergence rate of the expansions themselves, but there are aspects related to the `philosophy' in the way an element is expanded, that may be different with respect to the classical basis expansion. In fact, in the classical basis expansion the dictionary (i.e., the basis) is fixed and ordered independently from the expanded element; the coefficients are then chosen depending on the element we want to represent, and each element of the basis appears only once in the expansion. In the greedy algorithms we still have a fixed set of expanding elements, but without a prescribed order, and in the expansion it could happen that an element from the dictionary appears more than once (still with coefficients that in general depend on the element we want to represent). In the greedy algorithms with prescribed coefficients, that constitute the object of the present work, both the dictionary and the coefficients are fixed, independently from the vector we want to expand, and the only thing that depends on the element we want to represent is how the coefficients are associated to the corresponding elements from the dictionary; when the algorithm converges, the convergence of the corresponding series expansion is in general not unconditional.

There is a large literature on greedy expansion, both in Hilbert and in Banach spaces. We refer in particular to the works of Temlyakov \cite{libroprincipale}, \cite{articolo1}, \cite{articolo2}, where a large amount of material on different kind of greedy algorithms can be found. The greedy expansions with prescribed coefficients in Banach spaces have been introduced in \cite{articolo3}. Here we study a version of such algorithm for real Hilbert spaces. Let then $H$ be a real Hilbert space; we indicate with $\langle\cdot,\cdot\rangle$ its inner product. We say that $\mathcal{D}\subset H$ is a \emph{dictionary} if
\begin{itemize}
	\item[(i)] for every $g\in\mathcal{D}$ we have $\Vert g\Vert=1$
	\item[(ii)] $\mathcal{D}$ is complete, in the sense that $\overline{\lc\{\mathcal{D}\}} = H$. 
\end{itemize}

We say that a dictionary $\mathcal{D}$ is symmetric if for every $g\in\mathcal{D}$ we have $-g\in\mathcal{D}$. For a generic dictionary $\mathcal{D}$ we indicate with $\mathcal{D}^{\pm}$ its symmetrized, i.e., $\D^\pm:=\bigcup_{g\in\D}\{g,-g\}$. We shall suppose throughout the paper that the dictionary $\mathcal{D}$ is symmetric.

\vskip0.2cm
\noindent\textbf{Greedy Expansion with prescribed coefficients.} Now we describe how the algorithm works. We fix a coefficient sequence $C=\{c_n\}_n$ with $c_n>0$ for every $n$, and a \emph{weakening sequence} $\tau$, that is, a sequence $\tau=\{t_k\}_{k=1}^{\infty}$ with $0< t_k\leqslant 1$ for every $k$. For $f\in H$ we define:
\[f_0:=f\]
\[G_0:=0;\]
then, for every $m\geqslant 1$ we procede as follows. If $f_{m-1}\neq 0$:
\begin{itemize}
	\item We choose an element $\varphi_m\in\D$ such that
	\begin{equation}\label{finale1}
	\langle f_{m-1},\varphi_m\rangle\geqslant t_m\sup_{g\in\D}\langle f_{m-1},g\rangle\; .
	\end{equation}
	\item We define
	\[
	f_m:=f_{m-1}-c_m\varphi_m
	\]
	\[
	G_m=G_{m-1}+c_m\varphi_m\; .
	\]
\end{itemize}
In this way, $G_m$ represents the approximation of $f$, and $f_m$ is the remainder of such an approximation at the step $m$. Then the greedy expansion of $f$ is the series
$$
\sum_{n=1}^\infty c_n\varphi_n.
$$
If $f_{m-1}=0$ then the algorithm stops, and the corresponding approximation at this step coincides with $f$. 

Since the choice of $\varphi_m$ satisfying \eqref{finale1} is not unique, we may have different expansions of the same element.

We observe that if the weakening sequence $\tau$ satisfies $t_k=1$ for some $k$, the greedy expansion could not exist, since at some step it could be impossible to find $\varphi_m$ satisfying \eqref{finale1}. On the other hand, the case $t_k<1$ is much more difficult to treat, while for $t_k=1$ there are more complete results in the literature. 
Moreover, even when the greedy expansion exists, it does not need to converge to $f$. Of course the interesting case is when this happens, so we say that the greedy algorithm is \emph{convergent} if for every $f\in H$ we have
\begin{equation*}
	\lim_{m\to\infty} G_m=f,
\end{equation*}
for every approximating sequence $\{G_m\}_{m=1}^\infty$ obtained by the greedy algorithm.

When the weakening sequence satisfies $t_k=t$ for every $k$ we say that the greedy algorithm has \emph{weakening parameter} $t$. In this paper we deal with greedy algorithms with prescribed coefficients and weakening parameter $1$, assuming that at least one greedy expansion always exists. In this case the convergence of the algorithm is characterized by the following two theorems, proved in \cite{articolo4}.

	\begin{thm}[\cite{articolo4}]\label{pc24}
	Let $\D$ be a symmetric dictionary in $H$, and $C$ be a sequence of positive coefficients satisfying
	\[
	\sum_{n=1}^\infty c_n=\infty
	\]
	\[
	c_n=o\left(\frac{1}{\sqrt{n}}\right),\;  n\rightarrow\infty\; .
	\]
	Then, for every $f\in H$, all the possible greedy expansions of $f$ with prescribed coefficients $C$ and weakening parameter $t=1$ with respect to the dictionary $\D$ converge to $f$.
\end{thm}
	
\begin{thm}[\cite{articolo4}]
	There exists a Hilbert space $H$, a symmetric dictionary $\D$, an element $f\in H$ and a monotonic sequence of positive coefficients $C=\{c_n\}_{n=1}^\infty$ with $c_n\asymp\frac{1}{\sqrt{n}}$, such that there is a unique realization of the greedy expansion of $f$ with prescribed coefficients $C$ and weakening parameter $t=1$, and such realization does not converge to $f$.
\end{thm}

The previous results leave open the question if the greedy algorithm can converge with weaker conditions on the coefficients for particular Hilbert spaces or for particular dictionaries. In fact, in \cite{articolo5} it is proved that in the case of finite dimensional spaces the greedy algorithm is convergent for a larger set of prescribed coefficients with respect to Theorem \ref{pc24}, as stated in the following result.
	
\begin{thm}[\cite{articolo5}]\label{arte}
		Let $H$ be a finite dimensional Hilbert space, $\D$  be a symmetric dictionary and $C$ be a non increasing sequence of positive coefficients such that
		\[
		\sum_{n=1}^\infty c_n=\infty
		\]
		\[
		\lim_{n\to\infty}c_n=0\; .
		\]
		Then for every $f\in H$, all the possible greedy expansions of $f$ with prescribed coefficients $C$ and weakening parameter $t=1$ converge to $f$.
	\end{thm}
\noindent In this paper we extend the previous results in several directions.
\begin{itemize}
	\item[-] Concerning the finite dimensional case, we prove that the result of Theorem \ref{arte} holds without requiring that $C$ is non increasing, and moreover it can be extended to an arbitrary weakening sequence $\tau=\{ t_k\}_{k=1}^\infty$ with $t_k\in(0,1]$ under the following conditions on the coefficients (cf. Theorem \ref{fin1})
	\begin{equation}\label{conditions}
	\sum_{n=1}^\infty c_nt_n=\infty,\qquad \lim_{n\to\infty}\frac{c_n}{t_n}=0.
	\end{equation}
\end{itemize}
Concerning the infinite dimensional case we prove the following results.
\begin{itemize}
	\item[-] When the dictionary is the symmetrized of an orthonormal basis and we consider a weakening parameter $t=1$ (i.e., $t_k=1$ for every $k$), the conditions \eqref{conditions} are sufficient to guarantee the convergence of the algorithm (cf. Theorem \ref{val2}).
	\item[-] The fact that in the previous point we consider $t=1$ is essential; indeed when the dictionary is the symmetrized of an orthonormal basis and we consider a weakening parameter $t<1$ (i.e., $t_k=t<1$ for every $k$) and coefficients satisfying \eqref{conditions}, the convergence of the algorithm fails (cf. Theorem \ref{val3}).
	\item[-] When the dictionary is the symmetrized of the union of an orthonormal basis $\mathcal{E}$ and an arbitrary number of elements that are linear combinations of a finite subset of $\mathcal{E}$, the convergence of the algorithm (with weakening parameter $t=1$) is guaranteed under the hypotheses \eqref{conditions} (cf. Corollary \ref{new3}).
\end{itemize}
The last result is a consequence of a general result on greedy expansions on direct sums of Hilbert spaces, given in Theorem \ref{sd8}.

\section{Finite-dimensional case}

In this section we prove a general result on convergence of greedy algorithm for finite dimensional spaces.

\begin{thm}\label{fin1}

	Let $H$ be a finite-dimensional Hilbert space, $\mathcal{D}$ be a dictionary in $H$, $C = \{c_n \}_{n=1}^{\infty}$ be a sequence of positive coefficients and $\tau = \{t_n \}_{n=1}^{\infty}$ be a sequence of coefficients in $(0, 1]$ such that 
	\begin{enumerate}
		\item $\sum\limits_{n=1}^{\infty} c_n t_n = \infty$.
		\item $\lim\limits_{n \to \infty} \cfrac{c_n}{t_n} = 0$.
	\end{enumerate}
	Then, for every $f \in H$, all the possible greedy expansions of $f$ with prescribed coefficients
	$C$ and weakening sequence $\tau$ with respect to the dictionary $\mathcal{D}$ converge to f.
	
\end{thm}
\begin{proof}
	
	Due to the fact that in the finite-dimensional Hilbert space a unit sphere is a compact and the completeness of $\mathcal{D}$, there exist a constant $c > 0$ such that for every $f \in H$
	\begin{align*}
		\sup\limits_{g \in D}\left\langle f, g \right\rangle > c \lVert f \rVert.
	\end{align*}
	Then from the definition of greedy expansions with prescribed coefficients we immediately have that
	\begin{align}\label{4.1}
		\left\langle f_n, \varphi_{n+1} \right\rangle > c t_{n+1}\lVert f_n \rVert.
	\end{align}
	Let us fix an arbitrary $\varepsilon > 0$. Due to the second condition of the theorem, there exist $N > 0$, such that for every $n > N$ the following inequalities hold 
	\begin{align}\label{4.2}
		\cfrac{c_n}{t_n} < \varepsilon{,} \qquad c_n < \cfrac{\varepsilon}{c}.
	\end{align}
	Let us assume that $\lVert f_k \rVert > \cfrac{\varepsilon}{c}$ for some $k > N$. In this case, we will show that there exists $l > k$ such that
	\begin{align}\label{4.3}
		\lVert f_l \rVert < \cfrac{\varepsilon}{c}.
	\end{align}
	Assume the contrary. Then from $\left(\ref{4.1}\right)$ and $\left(\ref{4.2}\right)$ for every $n > k$ we get
	\begin{align}
		\lVert f_{n+1} \rVert ^2 &= \langle f_{n+1}, f_{n+1}\rangle = \langle f_n - c_{n+1}\varphi_{n+1}, f_n - c_{n+1}\varphi_{n+1}\rangle = \notag \\
		&= \lVert f_{n} \rVert ^2 - 2c_{n+1}\langle f_n, \varphi_{n+1}\rangle + c_{n+1}^2 \leqslant \notag \\
		&\leqslant \lVert f_{n} \rVert ^2 - 2c_{n+1}t_{n+1}c\lVert f_{n} \rVert + c_{n+1}^2 = \label{agg} \\
		&=  \lVert f_{n} \rVert ^2 - c_{n+1}t_{n+1}c\lVert f_{n} \rVert - c_{n+1}(t_{n+1}c\lVert f_{n} \rVert - c_{n+1}) \leqslant \notag \\
		&\leqslant \lVert f_{n} \rVert ^2 - c_{n+1}t_{n+1}c\lVert f_{n} \rVert - c_{n+1}t_{n+1}(c\lVert f_{n} \rVert - \varepsilon) \leqslant \notag \\
		&\leqslant \lVert f_{n} \rVert ^2 - c_{n+1}t_{n+1}\varepsilon. \label{4.4}
	\end{align}
	It implies 
	\begin{align*}
		\lVert f_{n+m} \rVert ^2 &\leqslant \lVert f_{n} \rVert ^2 - \varepsilon\sum\limits_{p=1}^{m}c_{n+p}t_{n+p} \underset{m \to \infty}{\longrightarrow} -\infty{,}
	\end{align*}
	which contradicts with the assumption. Then for some $l > k$ we have $(\ref{4.3})$. Now, if the inequality  
	\begin{align*}
	\lVert f_{n} \rVert < \cfrac{\varepsilon}{c}.
	\end{align*} 
	holds for all $n \geqslant l$, the proof of the theorem is complete. So let us suppose it exists $j$ such that
	\begin{align*}
	\lVert f_{j} \rVert \geqslant \cfrac{\varepsilon}{c}.
	\end{align*} 
	and $j$ is the first step for which the above inequality holds after $l$. We observe that from \eqref{agg} we get
	\begin{align*}
	\lVert f_{j} \rVert ^2 &\leqslant \lVert f_{j-1} \rVert ^2 - 2c_{j}t_{j}c\lVert f_{j-1} \rVert + c_{j}^2 < \notag \\
	&< \cfrac{\varepsilon^2}{c^2} + \cfrac{\varepsilon^2}{c^2} = \notag \\ &= 2\,\cfrac{\varepsilon^2}{c^2}.
	\end{align*}
	From now on, $(\ref{4.4})$ holds for the sequential steps $n>j$ till a possible step $d$ (which exists) such that
	$\lVert f_{d} \rVert < \cfrac{\varepsilon}{c}$ , so during all these steps the sequence is not increasing. Moreover, since the above pattern is replicable after step $d$, we can finally deduce that 
	\begin{align*}
	\lVert f_n \rVert < \sqrt{2}\, \cfrac{\varepsilon}{c}
	\end{align*}
	for all $n\geqslant l$, and this completes the proof of the theorem.
\end{proof}

\section{Symmetrized orthonormal bases}
	
In this section we analyze the case when the dictionary is the symmetrized of an orthonormal basis (in infinite dimensional spaces), proving that when the weakening parameter is $t=1$, the hypotheses on the coefficients $C$ given in Theorem \ref{fin1} are enough to guarantee the convergence of the algorithm, while for $t<1$ the convergence fails, in the sense that there exist an example where greedy expansion does not converge.

\begin{thm}\label{val1}
    Let $\mathcal{E} = \{e_i\}_{i=1}^{\infty}$ be an orthonormal basis for a (separable) Hilbert space $H$. We consider the following dictionary for $H$
    $$
    \mathcal{D} := \mathcal{E}^\pm = \{e_i\}_{i=1}^{\infty} \cup \{-e_i\}_{i=1}^{\infty}.
    $$
    Let moreover $C$ be a sequence of positive coefficients such that
    \begin{gather*}
        \sum_{p=1}^{\infty} c_p = \infty \\
        \lim_{p \to \infty} c_p = 0.
    \end{gather*}
    Then for every $f \in H$, all the possible greedy expansions of $f$ with prescribed coefficients $C$ and weakening parameter $t \leqslant 1$ satisfy $\liminf_{n \to \infty} ||f_n|| = 0.$
\end{thm}

\begin{proof}
In the following we write for simplicity the dictionary as
$$
\mathcal{D} = \{e_i\}_{i\in\Z\setminus\{ 0\}}
$$
where we intend $e_{-i} = -e_i$ for every $i \geqslant 1$. Since $\{e_i\}_{i=1}^{\infty}$ is an orthonormal basis, we can write $f$ as $$f = \sum_{i=1}^{\infty} x_i e_i$$ for a unique choice of the coefficients $x_i$. Without loss of generality we can assume that the sequence $\{|x_i|\}_{i=1}^{\infty}$ is non increasing. Similarly, for every remainder $f_n$ we can write $$f_n = \sum_{i=1}^{\infty} x_{i, n} e_i$$ for a unique choice of the coefficients $x_{i, n}$. For convenience we put $x_{i, 0} = x_i$.

We notice that the algorithm, at each step $m$, modifies one single coefficient of the orthonormal expansion of the remainder $f_{m-1}$.

Let $e_{k_m}$ be an element from the dictionary that is touched by the algorithm at step $m$. We note that  

\begin{align*}
    \lVert f_n \rVert^2 &= \langle f_n, f_n\rangle = \langle f_{n-1} - c_n e_{k_n}, f_{n-1} - c_n e_{k_n}\rangle = \\
    &= \lVert f_{n-1} \rVert^2 - 2c_n\langle f_{n-1}, e_{k_n}\rangle + c_n^2 = \ldots =\\
    &= \lVert f_0 \rVert^2 - 2\sum_{m=1}^n c_m |x_{m-1, |k_m|}| + \sum_{m=1}^n c_m^2 \leqslant \\
    &\leqslant \lVert f_0 \rVert^2 - 2t\sum_{m=1}^n c_m \sup_{e \in \mathcal{D}} \langle f_{m-1}, e\rangle + \sum_{m=1}^n c_m^2,
\end{align*}
which means that 
\begin{align}\label{1.1}
    \liminf_{n \to \infty}\sup_{e \in \mathcal{D}} \langle f_n, e\rangle = 0
\end{align}
due to conditions of Theorem \ref{val1}.

As $\{x_i\}_{i=1}^{\infty} \in l^2$, we have that $\liminf_{n \to \infty} \sqrt{n} |x_n| = 0$.

Let now $\varepsilon$ be an arbitrary positive number, and let $n$ be such that $|x_n| \leqslant \frac{\varepsilon}{\sqrt{n}}$.

Also, let $k$ be the first step such that $$|x_n| = |x_{n, k-1}| \geqslant t \sup_{e \in \mathcal{D}} \langle f_{k-1}, e\rangle;$$ such a number exists due to (\ref{1.1}). Then $x_n$ has not been modified in the first $k-1$ steps; since $\{ |x_i|\}_{i=1}^\infty$ is non increasing, also $x_m$, $m\geqslant n$, has not been modified in the first $k-1$ steps, and so $x_{m,k-1}=x_m$ for every $m\geqslant n$. Now, for every $m$ we have the following estimation $$|x_{m, k-1}| \leqslant \sup_{e \in \mathcal{D}} \langle f_{k-1}, e\rangle \leqslant \frac{|x_n|}{t} \leqslant \frac{\varepsilon}{t\sqrt{n}}.$$ Therefore, using the monotonicity of the sequence $\{|x_i|\}_{i=1}^{\infty}$, we have
\begin{align*}
    ||f_{k-1}||^2 &= \sum_{m=1}^n x_{m, k-1} ^ 2 + \sum_{m=n+1}^\infty x_{m, k-1} ^ 2 \leqslant \\
    &\leqslant \frac{\varepsilon^2}{t^2} +  \sum_{m=n+1}^\infty x_{m}^2.
\end{align*}
Since $\varepsilon>0$ is arbitrary and $n \rightarrow \infty$, we get that $\liminf_{n \to \infty} ||f_n|| = 0$, which completes the proof of Theorem \ref{val1}.

\end{proof}

\begin{thm}\label{val2}
    Let $\mathcal{E} = \{e_i\}_{i=1}^{\infty}$ be an orthonormal basis for a (separable) Hilbert space $H$. We consider the following dictionary for $H$ $$\mathcal{D} := \mathcal{E}^\pm = \{e_i\}_{i=1}^{\infty} \cup \{-e_i\}_{i=1}^{\infty}.$$ Let moreover $C$ be a sequence of positive coefficients such that
    \begin{gather*}
        \sum_{p=1}^{\infty} c_p = \infty \\
        \lim_{p \to \infty} c_p = 0.
    \end{gather*}
    Then for every $f \in H$, all the possible greedy expansions of $f$ with prescribed coefficients $C$ and weakening parameter $t = 1$ converge to $f$.
\end{thm}

\begin{proof}
The proof is the direct implication of Theorem \ref{val1} and Lemma 4.3 from \cite{articolo4}.
\end{proof}

\begin{rem}
	The result of Theorem \ref{val2} holds also for finite dimensional spaces, since in the finite dimensional case the more general result of Theorem \ref{fin1} holds.
\end{rem}

\begin{thm}\label{val3}
Let $H$ be a (separable, infinite dimensional) Hilbert space, $\mathcal{D}$ a dictionary as in Theorem \ref{val2}, and $t\in(0,1)$. There exist an element $f \in H$ and a sequence $c_n \rightarrow 0$, with $\sum_{n=1}^\infty c_n=\infty$, such that a greedy expansion of $f$ in the dictionary $\mathcal{D}$ with the prescribed coefficients $\{c_n\}_{n=1}^{\infty}$ and weakening parameter $t$ does not converge to $f$.
\end{thm}

\begin{proof}

Let $k\in\N$, with $k>1$, be such that $t^k<\frac{1}{\sqrt{k}}$, and consider the element $f\in H$ whose orthonormal basis components are
\begin{align*}
\bigl( \underbrace{t^k, \dots,t^k}_{k\ \text{times}}, \dots, \underbrace{t^{k+j},\dots,t^{k+j}}_{k+j\ \text{times}}, ... \bigr),
\end{align*}
where the components are subdivided in groups, and the $j$-th group contains $k+j$ components, all equal to $t^{k+j}$.

In order to prove the theorem, we build the sequence $C = \{c_n\}_{n=1}^{\infty}$ so that the following conditions hold:
\begin{itemize}
    \item[(i)] $\limsup_{n \to \infty}||f_n|| \geqslant 1$,
    \item[(ii)] $c_n \rightarrow 0 \quad (n \rightarrow \infty)$,
    \item[(iii)] $\sum\limits_{n=1}^{\infty}c_n = \infty$.
\end{itemize} 

We construct $c_n$ consequently for each group, starting from the first one, in the following way:\\
--- Increase the subnorm of the group to one.\\
--- Make the subnorm of the group equal to zero.\\
--- Move to the next group

Consider a group of components
$$
\bigl(\underbrace{t^h,\dots,t^h}_{h\ \text{times}}\bigr),\quad h\geqslant k.
$$
First, we take any component of the group and choose $c_i$ in the coefficient sequence $C$ in order that the component changes its sign and is multiplied by $\frac{1}{t}$ (so, $t^h$ becomes $-t^{h-1}$); then we do the same for the remaining elements of the group (it is possible since we are applying the algorithm with weakening parameter $t$). We repeat this procedure on the group until each component of the group in the remainder has modulus between $\frac{t}{\sqrt{h}}$ and $\frac{1}{\sqrt{h}}$.

At this point we choose the next coefficients to consequently change all the components of the group to $\frac{1}{\sqrt{h}}$ (or $-\frac{1}{\sqrt{h}}$), that means that at this step the subnorm of the group equals one, and so condition (i) holds.

Finally we choose the next coefficients equal to $\frac{1}{\sqrt{h}}$ in order that all the components of the group become $0$. Then we pass to the next group of components and repeat the same procedure.

We observe that all the coefficients $c_i$, in these steps, are less than $\frac{2}{\sqrt{h}}$, and since $h\to\infty$ we have that condition (ii) holds. Moreover, since for every $h$ we choose at least one coefficient equal to $\frac{1}{\sqrt{h}}$ we have that condition (iii) is satisfied, and this completes the proof.
\end{proof}

\section{Direct sum of Hilbert spaces}

In this section we prove that the greedy algorithm with prescribed coefficients in infinite dimensional Hilbert spaces is still convergent (with the same hypotheses on the coefficients as in Theorem \ref{val2}) if we consider as a dictionary the symmetrized of the union of an orthonormal basis $\mathcal{E}$ and an arbitrary number of elements that are linear combinations of a finite subset of $\mathcal{E}$. In order to do this we prove a result on greedy algorithm in direct sums of Hilbert spaces.
 
Recall that, given two Hilbert spaces $H_1$ and $H_2$ their direct sum is defined as
\[
H_1\oplus H_2=\{(x_1,x_2): x_1\in H_1,\, x_2\in H_2\}\; ;
\]
$H_1\oplus H_2$ is a Hilbert space with inner product given by
\[
\langle (x_1,y_1),(x_2,y_2)\rangle_{H_1\oplus H_2}:=\langle x_1,x_2\rangle_{H_1}+\langle y_1,y_2\rangle_{H_2}\; .
\]
Let $H$ be a Hilbert space and $M\subset H$ a closed subspace of $H$. It is well-known that
\begin{equation}\label{isom-projection}
	H\cong M\oplus M^\perp,
\end{equation}
where $M^\perp\subset H$ is the orthogonal of $M$ in $H$. Indeed, from the Projection Theorem we have that each $x\in H$ can be uniquely decomposed as $x=Px+Qx$, with $Px\in M$ and $Qx\in M^\perp$. Then the map
	\begin{equation}\label{isom-alpha}
	\begin{split}
	\alpha\, :\;& H\rightarrow M\oplus M^\perp\\
	&x\mapsto (Px,Qx)
	\end{split}
	\end{equation}
is an isometric isomorphism.

\noindent Of course we can consider the direct sum of a finite number of Hilbert spaces.

Now we want to prove that the convergence of the greedy algorithm is maintained by isometric isomorphisms, in the following sense.

\begin{lem}\label{sd9}
	Let $H$ be a Hilbert space, $\D$ a dictionary for $H$, and $C$ a prescribed sequence of coefficients. We suppose that, for every $f\in H$, all the possible greedy expansions of $f$ with respect to the dictionary $\D$ with prescribed coefficients $C$ and weakening parameter $t=1$ converge to $f$. Let moreover $K$ be another Hilbert space and $\alpha: H\rightarrow K$ be an isometric isomorphism. Define $\mathcal{E}:=\{\alpha(\varphi):\varphi\in\D\}$. Then $\mathcal{E}$ is a dictionary for $K$ (symmetric if $\D$ is symmetric), and for every $g\in K$ all the possible greedy expansions of $g$ with respect to the dictionary $\mathcal{E}$ with prescribed coefficients $C$ and weakening parameter $t=1$ converge.
\end{lem}
\begin{proof}
	The fact that $\mathcal{E}$ is a dictionary (symmetric if $\D$ is symmetric) is straightforward since $\alpha$ is an isometric isomorphism.
	
	Now let $g\in K$ and consider a greedy expansion of $g$, that is an expression of the kind
	\[
	\sum_{p=1}^\infty c_p\psi_p
	\]
	with $\psi_p:=\alpha(\varphi_p)\in\mathcal{E}$. We observe that at the $m$-th step the greedy approximation is obtained by choosing $\psi_m\in\D$ in such a way that
	\[
	\langle g-\sum_{p=1}^{m-1}c_p\psi_p,\psi_m\rangle = \sup_{\psi\in\mathcal{E}}\langle g-\sum_{p=1}^{m-1}c_p\psi_p,\psi\rangle\; .
	\]
	Applying $\alpha^{-1}$ to both sides we then obtain
	\begin{equation}\label{sd1}
	\langle f-\sum_{p=1}^{m-1}c_p\varphi_p,\varphi_m\rangle = \sup_{\varphi\in\D}\langle f-\sum_{p=1}^{m-1}c_p\varphi_p,\varphi\rangle\; .
	\end{equation}
	with $f=\alpha^{-1}(g)$. This holds for every positive integer $m$, so by definition it follows that
	\[
	\sum_{p=1}^\infty c_p\varphi_p
	\]
	is a greedy expansion of $f$ with prescribed coefficients $C$ with respect to the dictionary $\D$. Then by hypotheses we know that
	\[
	\lim_{m\to \infty}\sum_{p=1}^m c_p\varphi_p=f
	\]
	and so, since $\alpha$ is an isometric isomorphism,
	\[
	\lim_{m\to\infty}\sum_{p=1}^m c_p\psi_p=g\; .
	\]
	Then the considered greedy expansion of $g$ converges to $g$, and the proof is completed.
\end{proof}

Now we analyze greedy expansions in direct sums of Hilbert spaces.

\begin{thm}\label{sd8}
	Let $H_j$ be a Hilbert space and $\D_j\subset H_j$ a symmetric dictionary for $H_j$, for $j\in\{1,\dots, N\}$. Suppose that for every $j$ and for every sequence of coefficients $C$ satisfying $\lim_{p\to\infty}c_p=0$ and $\sum_{p=1}^\infty c_p=\infty$, we have that for every $h\in H_j$ all the possible greedy expansions of $h$ with respect to the dictionary $\D_j$ with prescribed coefficients $C$ and weakening parameter $t=1$ converge. Let moreover
	$$
	H=H_1\oplus H_2\oplus\dots\oplus H_N,
	$$
	and consider in $H$ the set
	\[
	\D=\{(\varphi^1,0,\dots,0)\}_{\varphi^1\in\D_1}\cup\{(0,\varphi^2,0,\dots,0)\}_{\varphi^2\in\D_2}\cup\dots\cup\{(0,\dots,0,\varphi^N)\}_{\varphi^N\in\D_N}\; .
	\]
	Then $\D$ is a symmetric dictionary for $H$; moreover for every $f\in H$ and for every sequence of coefficients $C$ as above, all the greedy expansions of $f$ with respect to the dictionary $\D$ with prescribed coefficients $C$ and weakening parameter $t=1$ converge to $f$.
\end{thm}

\begin{proof}
	Let $f\in H$. Then $f=(f^1,\dots,f^N)$ with $f^j\in H_j$ for every $j\in\{1,\dots, N\}$. Since a generic element $\varphi\in\D$ is of the form
	\[
	\varphi=(0,\dots,0,\varphi^l,0,\dots,0)
	\] 
	with $l\in\{1,\dots,N\}$ and $\varphi^l\in\D_l$, we have
	\begin{equation}\label{sd4}
	\langle f, \varphi\rangle_H=\langle (f^1,\dots,f^N), (0,\dots,0,\varphi^l,0,\dots,0)\rangle_H=\langle f^l, \varphi^l\rangle_{H_l}\; .
	\end{equation} 
	The inner product in $H$ between $f$ and an element of $\D$ is then the inner product in $H_l$ of the unique non-zero component of the element of $\D$ with the corresponding component of $f$.
	
	We verify first that $\D$ is a symmetric dictionary. The fact that $\D$ is symmetric and that its elements have norm $1$ is straightforward. We then have to prove that $\D$ is complete; let $f\in H$ be such that
	\[
	\langle f,\varphi\rangle_H=0
	\]
	for every $\varphi\in\D$. By the observation above we then have
	\[
	\langle f^j, \varphi^j\rangle_{H_j}=0
	\]
	for every $j\in\{1,\dots,N\}$ and $\varphi^j\in\D_j$. Since each $\D_j$ is complete in $H_j$, we deduce that $f^j=0$ for every $j$, and so $f=0$. We then have that $\D$ is a dictionary.
	
	Observe now that at the $m$-th step of the greedy algorithm with prescribe coefficients $C$ applied to $f\in H$ we look for an element $\varphi_m\in\D$ such that
	\begin{equation*}
	\langle f_{m-1},\varphi_m\rangle_H=\sup_{\varphi\in\D}\langle f_{m-1},\varphi\rangle_H\; ,
	\end{equation*}
	where $f_{m-1}=(f_{m-1}^1,\dots,f_{m-1}^N)$ is the remainder of the greedy approximation at the $(m-1)$-th step. The element $\varphi_m$ selected by the algorithm is of the form
	\[
	\varphi_m=(0,\dots,0,\varphi_m^{j(m)},0,\dots,0)\; ,
	\] 
	with $j(m)\in\{1,\dots,N\}$ and $\varphi_m^{j(m)}\in\D_{j(m)}$. Then at the $m$-th step the remainder $f_{m-1}$ is modified only in its $j(m)$-th component. More precisely we obtain
	\[
	f_m=(f_m^1,\dots,f_m^N)=(f_{m-1}^1,\dots,f_{m-1}^{j(m)-1},f_{m-1}^{j(m)}-c_m\varphi_m^{j(m)},f_{m-1}^{j(m)+1},\dots,f_{m-1}^N)\; .
	\]
	Then, when we write the greedy expansion $\sum_{m=1}^\infty c_m\varphi_m$ of $f$, the coefficients in $C$ are distributed in the various components of the direct sum; so there exist $N$ subsequences $\{c_{n_p(l)}\}_p$ of $C$, for $l\in\{1,\dots,N\}$, such that the greedy expansion of $f$ is of the form
	\begin{equation}\label{add-vectorexp}
	\left(\sum_{p}c_{n_p(1)}\varphi_{p}^1,\dots,\sum_{p}c_{n_p(N)}\varphi_{p}^N\right)\; ,
	\end{equation}
	with $\varphi_p^l\in\D_l$ for every $l=1,\dots,N$.
	We observe that some of the sequences $\{c_{n_p(l)}\}_{p}$ may be finite (and then the corresponding component in \eqref{add-vectorexp} is a finite sum), or may even be the empty set (in this case we mean that the corresponding component in \eqref{add-vectorexp} is $0$). Moreover the following properties are satisfied:
	\[
	\bigcup_{l=1}^N\{n_p(l)\}_{p}\subseteq \N,
	\]
	\[ \forall l\neq k,\;\{n_p(l)\}_{p}\cap\{n_p(k)\}_{p}=\emptyset\; .
	\]
	
	Our aim is to prove that
	\[f=\sum_{m} c_m\varphi_m\; ,\]
	so we have to prove that for every $l\in\{1,\dots,N\}$ the series (or finite sum) $\sum_{p}c_{n_p(l)}\varphi_{p}^l$ converges to $f^l$. Observe that for every $l$ the expression $\sum_pc_{n_p(l)}\varphi_p^l$ is a greedy expansion of $f^l$ with respect to the dictionary $\D_l$ with prescribed coefficients $\{c_{n_p(l)}\}_p$. We have two different cases.
	
	Suppose first that
	\[
	\bigcup_{l=1}^N\{n_p(l)\}_{p}\subsetneq \N\; .
	\]
	Then for every $l\in\{1,\dots,N\}$ the sequence $\{c_{n_p(l)}\}_{p}$ is finite, and this means that at a certain step $m$ the greedy algorithm stops, and this happens only if $f_{m-1}=0$. In this case the greedy expansion is a finite sum that equals $f$, and so the proof is complete.
	
	We have to analyze the case
	\[
	\bigcup_{l=1}^N\{n_p(l)\}_{p}= \N\; .
	\]
	Since by hypothesis $\sum_{m=1}^\infty c_m=\infty$ and the coefficients in $C$ are positive, there exists at least an index $l'$ such that $\sum_{p}c_{n_p(l')}=\infty$. Moreover, since $\lim_{m\to\infty}c_m=0$, we have that $\{c_{n_p(l')}\}_{p}$ satisfies $\lim_{p\to\infty}c_{n_p(l')}=0$. Then by hypothesis the greedy expansion \[\sum_{p}c_{n_p(l')}\varphi_p^{l'}\] converges to $f^{l'}$. Now let us suppose that the greedy expansion of $f$ does not converge to $f$, i.e., there exists $l''$ such that
	\begin{equation}\label{sd2}
	\sum_{p}c_{n_p(l'')}\varphi_p^{l''} 
	\end{equation}
	does not converge to $f^{l''}$. We have to distinguish two cases, since $\{n_p(l'')\}_p$ can be an (infinite) sequence or a finite set.
	
	Suppose first that $\{n_p(l'')\}_p$ is an (infinite) sequence. We have already observed that $\{c_{n_p(l'')}\}_{p}$ satisfies $\lim_{p\to\infty}c_{n_p(l'')}=0$. Then we must have
	\begin{equation}\label{sd3}
	\sum_{p}c_{n_p(l'')}<\infty\; ,
	\end{equation}
	otherwise the corresponding greedy expansion would converge to $f^{l''}$ by hypothesis. Moreover we observe that, since all the elements in the dictionary $\D$ have unitary norm, \eqref{sd3} implies
	\[
	\sum_p\Vert c_{n_p(l'')}\varphi_p^{l''}\Vert=\sum_p\vert c_{n_p(l'')}\vert=\sum_p c_{n_p(l'')}<\infty\; .
	\]
	The series \eqref{sd2} is then absolutely convergent, and so there exists $g\in H_{l''}$ with $g\neq f^{l''}$, such that the greedy expansion $\sum_{p}c_{n_p(l'')}\varphi_p^{l''}$ converges to $g$. Then
	\[
	\lim_{m\to\infty} \left(f^{l''}-\sum_{p=1}^m c_{n_p(l'')}\varphi_p^{l''}\right)=f^{l''}-g\neq 0\; .
	\]
	Since $\D_{l''}$ is complete in $H_{l''}$, there exists $\tilde{\varphi}^{l''}\in\D_{l''}$ such that
	\[
	\langle f^{l''}-g,\tilde{\varphi}^{l''}\rangle_{H^{l''}}\neq 0\; .
	\]
	By the continuity of the inner product we have
	\[
	\lim_{m\to\infty}\langle f^{l''}-\sum_{p=1}^mc_{n_p(l'')}\varphi_p^{l''},\tilde{\varphi}^{l''}\rangle_{H^{l''}} =\langle f^{l''}-g,\tilde{\varphi}^{l''}\rangle_{H^{l''}}\neq 0\; ,
	\]
	so there exists $\epsilon''>0$ and an integer $m''>0$ such that for every $m>m''$
	\begin{equation}\label{sd5}
	\vert\langle f^{l''}-\sum_{p=1}^m c_{n_p(l'')}\varphi_p^{l''},\tilde{\varphi}^{l''}\rangle_{H^{l''}}\vert\geqslant\epsilon''\; .
	\end{equation}
	On the other hand, since the greedy expansion $\sum_{p}c_{n_p(l')}\varphi_p^{l'}$ converges to $f^{l'}$, we have
	\[
	\lim_{m\to\infty} \left(f^{l'}-\sum_{p=1}^m c_{n_p(l')}\varphi_p^{l'}\right)=0\; .
	\]
	Then for every $\varphi^{l'}\in\D_{l'}$ we obtain
	\[
	\left\vert\langle f^{l'}-\sum_{p=1}^mc_{n_p(l')}\varphi_p^{l'},\varphi^{l'}\rangle_{H^{l'}}\right\vert \leqslant \left\Vert f^{l'}-\sum_{p=1}^mc_{n_p(l')}\varphi_p^{l'}\right\Vert \to 0
	\]
	for $m\to\infty$, and so for every $\epsilon>0$ there exists an integer $m'$ such that for every $m>m'$ and for every $\varphi^{l'}\in\D_{l'}$
	\begin{equation}\label{sd6}
	\vert\langle f^{l'}-\sum_{p=1}^m c_{n_p(l')}\varphi_p^{l'},\varphi^{l'}\rangle_{H_{l'}}\vert < \epsilon\; .
	\end{equation}
	Fix now $\epsilon<\epsilon''$. From \eqref{sd5} and \eqref{sd6} we deduce that for every $\varphi^{l'}\in\D_{l'}$ and for every $m>\max\{m',m''\}$ we have
	\begin{equation*}
	\vert\langle f^{l'}-\sum_{p=1}^m c_{n_p(l')}\varphi_p^{l'},\varphi^{l'}\rangle_{H_{l'}}\vert < \epsilon<\epsilon''\leqslant\vert\langle f^{l''}-\sum_{p=1}^m c_{n_p(l'')}\varphi_p^{l''},\tilde{\varphi}^{l''}\rangle_{H_{l''}}\vert\; .
	\end{equation*}
	Since the inner product between $f_{m-1}$ and an element of $\D$ is given by $\eqref{sd4}$, we have that starting from the step $\max\{m',m''\}$, the greedy algorithm does not select any more an element $\varphi_m$ of the dictionary $\D$ having as only non vanishing component the $l'$-th one. Then the series
	\[
	\sum_p c_{n_p(l')}\varphi_p^{l'}
	\]
	is a finite sum, but this contradicts the fact that $\sum_p c_{n_p(l')}=\infty$. The contradiction comes from the fact that we have assumed the existence of $l''$ such that the series $\sum_p c_{n_p(l'')}\varphi_p^{l''}$ does not converge to $f^{l''}$; then we have completed the proof, at least in the case when $\{n_p(l'')\}_p$ is an (infinite) sequence.
	
	We still have to consider the case when $\{n_p(l'')\}_p$ is a finite set; suppose that such a set contains $q$ elements. Then
	\begin{equation}\label{finite-l-second}
	\sum_{p=1}^q c_{n_p(l'')}\varphi_p^{l''}=g\in H_{l''}\; ,
	\end{equation}
	and we are assuming that $g\neq f^{l''}$.
	As above there exists $\tilde{\varphi}^{l''}$ and $\epsilon''>0$ such that
	\begin{equation*}
	\vert\langle f^{l''}-g,\tilde{\varphi}^{l''}\rangle_{H^{l''}}\vert\geqslant\epsilon''\; .
	\end{equation*}
	We can suppose without loss of generality that $l''$ is such that $|\langle f^{l''}-g,\tilde{\varphi}^{l''}\rangle|$ is maximum, among all the indexes $l''$ such that $\{n_p(l'')\}$ is finite.
	By \eqref{sd6}, choosing $\epsilon<\epsilon''$, we conclude that for every $m>m'$ and for every $\varphi^{l'}\in\D_{l'}$
	\begin{equation}\label{lastadd}
	\vert\langle f^{l'}-\sum_{p=1}^m c_{n_p(l')}\varphi_p^{l'},\varphi^{l'}\rangle_{H_{l'}}\vert <\vert\langle f^{l''}-g,\tilde{\varphi}^{l''}\rangle_{H_{l''}}\vert\, .
	\end{equation}
	Now, $\{n_p(l')\}_p$ is an infinite sequence, the right-hand side of \eqref{lastadd} is maximum among all $l''$ such that $\{ n_p(l'')\}_p$ is finite and we have already proved that if $\{ n_p(l'')\}_p$ is infinite we have convergence; then at some step the greedy algorithm will select an element of the dictionary $\D$ whose non vanishing component is the $l''$-th one. This would add a new term in the sum \eqref{finite-l-second}, and so $\{n_p(l'')\}_p$ would contain $q+1$ elements. Also in this case we have a contradiction, due to the assumption that there exists $l''$ such that $\sum_p c_{n_p(l'')}\varphi_p^{l''}$ is different from $f^{l''}$. Then the greedy expansion of $f$ converges to $f$ and the proof is complete.
\end{proof}
\begin{cor}\label{new3}
	Let $H$ be a Hilbert space and $\mathcal{E}$ an orthonormal basis of $H$. Consider a finite set $\mathcal{E}'\subset\mathcal{E}$, and let $Y=\{y_j\}_{j}$ be a set of (unitary norm) elements belonging to $\lc\{\mathcal{E}'\}$. Consider the dictionary
	\[
	\D=\mathcal{E}^{\pm}\cup Y^{\pm}
	\]
	where $\mathcal{E}^{\pm}$ and $ Y^{\pm}$ are the symmetrized of $\mathcal{E}$ and $Y$, respectively. Let moreover $C$ be a sequence of prescribed coefficients such that
	\[
	\lim_{p\to\infty}c_p=0\; ,\quad \sum_{p=1}^\infty c_p=\infty\; .
	\]
	Then for every $f\in H$ all the possible greedy expansions of $f$ with respect to the dictionary $\D$ with prescribed coefficients $C$ and weakening parameter $t=1$ converge to $f$.
\end{cor}
\begin{proof}
	Let $M\subset H$ be the finite dimensional subspace generated by $\mathcal{E}'$, and $M^\perp\subset H$ the orthogonal of $M$. By \eqref{isom-projection} we know that $H\cong M\oplus M^\perp$, with corresponding isomorphism given by \eqref{isom-alpha}. Consider the symmetrized $(\mathcal{E}')^{\pm}$ of $\mathcal{E}'$, and define
	\[
	\D_M=(\mathcal{E}')^{\pm}\cup Y^{\pm}\; ,
	\]
	\[
	\D_{M^\perp}=\mathcal{E}^{\pm}\backslash(\mathcal{E}')^\pm\; ;
	\]
	$\D_M$ is a symmetric dictionary for $M$ and $\D_{M^\perp}$ is a symmetric dictionary for $M^\perp$. By Theorem \ref{fin1} (applied to $M$) and Theorem \ref{val2} (applied to $M^\perp$) we have that the hypotheses of Theorem \ref{sd8} are satisfied, so for every $f\in M\oplus M^\perp$, all the possible greedy expansions of $f$ with respect to the dictionary
	\begin{equation}\label{sd10}
	\{(\varphi^M,0)\}_{\varphi^M\in\D_M}\cup\{(0,\varphi^{M^\perp})\}_{\varphi^{M^\perp}\in\D_{M^\perp}}
	\end{equation}
	with prescribed coefficients $C$ and weakening parameter $t=1$ converge. We then conclude from Lemma \ref{sd9}, since the isomorphism $\alpha: H\rightarrow M\oplus M^\perp$ given by \eqref{isom-alpha} transforms $\D$ in the dictionary \eqref{sd10}.
\end{proof} 

\section{Conclusions}
In this paper we have extended the convergence result of the greedy algorithm with prescribed coefficients known for finite dimensional Hilbert spaces, and we have proved that such a result, in the case of a weakening parameter $t=1$, holds also in the case of infinite dimensional spaces when the dictionary is the symmetrized of an orthonormal basis, to which linear combinations of elements from a finite subset of the basis itself can eventually be joined. These are sufficient conditions on the dictionary to have convergence of the algorithm; the question to characterize the dictionaries for which the greedy algorithm converges under the considered hypotheses on the coefficients remains open.

\section*{Acknowledgements}

The work of Al.R. Valiullin on Section 3 was supported by the Russian Science Foundation (project no. 21-11-00131) at Lomonosov Moscow State University. A. Oliaro was partially supported by the INdAM--GNAMPA project CUP\_E53C22001930001.

\end{document}